\documentclass[12pt]{amsart}

\usepackage{amsfonts}

\usepackage{amssymb, amsmath, amsthm}

\usepackage[french,english]{babel}

\newtheorem{thm}{Theorem}
\newtheorem{lem}{Lemma}[section]
\newtheorem{prop}[lem]{Proposition}
\newtheorem{cor}[lem]{Corollary}
\newtheorem*{defi}{Definition}

\newcommand{\C}{{\bf{C}}}
\newcommand{\R}{{\bf{R}}}
\newcommand{\Q}{{\bf{Q}}}
\newcommand{\Z}{{\bf{Z}}}

\newcommand{\newabstract}[1]{%
  \par\bigskip
  \csname otherlanguage*\endcsname{#1}%
  \csname captions#1\endcsname
  \item[\hskip\labelsep\scshape\abstractname.]
}

\begin{document}
 
\title[Asynchronous rotations]{Unique ergodicity of asynchronous rotations, and application}
\author{Fran\c cois Maucourant}
\address{Universit\'e Rennes I, IRMAR UMR 6625, Campus de Beaulieu 35042 Rennes cedex -  France}
\email{francois.maucourant@univ-rennes1.fr}
\thanks{Universit\'e de Rennes 1 et CNRS (UMR 6625)}


\maketitle
\begin{abstract} The main result of this paper is an analogue for a continuous family of tori of Kronecker-Weyl's unique ergodicity of irrational rotations. We show that the notion corresponding in this setup to irrationality, namely asynchronicity, is satisfied in some homogeneous dynamical systems. This is used to prove the ergodicity of naturals lifts of invariant measures. 
%\end{abstract}
\newabstract{french} Nous \'etudions sur une famille continue de tores les rotations dites asynchrones, analogues aux rotations irrationnelles sur les tores classiques. Le r\'esultat principal est l'unique ergodicit\'e de ces rotations sur un mono\"\i de adapt\'e. Nous prouvons que la condition d'asynchronicit\'e est v\'erifi\'ee dans une famille d'exemples issue de la dynamique homog\`ene, ce qui nous permet de d\'eduire l'ergodicit\'e de relev\'es de certaines transformations dans des fibr\'es en tores. 
\end{abstract}

\section{Introduction} 
 
\subsection{Motivations}

 The original motivation of this study  is the inquiry of ergodic properties of torus extension of homogeneous dynamical systems. Such dynamics have drawn some attention recently - see for example \cite{MR2060024}, \cite{MR2811599}, \cite{MR3332893} for unipotent actions, and \cite{MR1870070} and \cite{MR2753608} for diagonal actions. \\

 As an informal example, consider a diagonal element $a \in \mathbf{SL}(d,\R)=G_0$ ($d\geq 2$) with positive diagonal entries, acting by left multiplication on the homogeneous space $\mathbf{SL}(d,\R)/\mathbf{SL}(d,\Z)=G_0/\Gamma_0$, and let $\mu$ be a probability invariant by $a$ and ergodic on $G_0/\Gamma_0$. The main interesting cases for our purpose occur when the measure $\mu$ is not algebraic. This dynamical system is a factor of the action of  $(a,\alpha)$, where $\alpha\in \R^d$ is arbitrary, by left multiplication on $(\mathbf{SL}(d,\R)\ltimes \R^d) /(\mathbf{SL}(d,\Z)\ltimes \Z^d)=G/\Gamma$. This latter space is a torus bundle above $G_0/\Gamma_0$. Amongst the possible $(a,\alpha)$-invariant measures that projects onto $\mu$, there is a particular one, denoted by $\lambda$, which decomposes into the Haar measure of tori on each fiber. It is natural to ask about its ergodicity with respect to $(a,\alpha)$.\\

 Following the classical Hopf argument (see \cite{MR2261076}), one is naturally led to inquire about ergodic properties of the strong stable foliation of $(a,\alpha)$. It turns out that this foliation contains the orbits of another action, namely the multiplication by $(e,\beta)$ on $G/\Gamma$, where $\beta\in \R^d$ is an eigenvector for $a$ associated to an eigenvalue $<1$. This action is an unipotent action, but since it is "vertical" (in the sense trivial in the factor $G_0/\Gamma_0$), Ratner's theory  yield in this case no more information than Kronecker-Weyl's uniform distribution on the torus. \\

 To visualize this action of $(e,\beta)$ on each fiber, one may think of it as a rotation by a fixed vector $\beta$ on a varying torus depending on the base-point. Here, we will prefer to think of it as the rotation by a varying vector $f_\beta(x)$ depending on the base-point $x$, on a fixed torus $\mathbf{T}^d$.
 One may hope in this situation that $f_\beta(x) \in \mathbf{T}^d$ is irrational for almost every $x$. It turns out that under appropriate assumptions, the rotations defined by $f_\beta$ satisfy a stronger property, namely {\em asynchronicity}.\\
 
 As we will see shortly, such rotations on torus bundle above a measured space like $(G_0/\Gamma_0,\mu)$ enjoy strong ergodic properties, enabling us in this setting to prove a unique ergodicity result. In some sense, this can be considered as a weak analogue of Furstenberg's unique ergodicity of horocyclic flow.\\
   
  Finally, we will return to the question of the ergodicity of $\lambda$ with respect to $(a,\alpha)$, and related mixing properties. \\
  
  These kind of fiber-wise system were also investigated independently by Damien Thomine \cite{Thomine}, using another point of view.

\subsection{Asynchronous rotations} 
  
 Let $(I,\mathcal{B}(I),\mathcal{L})$ be a standard probability space without atoms, and let $K$ be a torus $\mathbf{T}^d=(\R/\Z)^d$, for some integer $d\geq 1$. 
 We think of a measurable map $f:I\to K$ as the data, for each $x\in I$, of a rotation adding the angle $f(x)$ in a torus indexed by $x$. Despite what the above motivational example might suggest, in this abstract setting, the case $d=1$ of rotations on a family of circles above a probability space is already interesting, and contains most of the difficulties. \\
 
 The set of such measurable maps $\{ f : I\to K \}$, where we identify two maps if they coincide $\mathcal{L}$-almost everywhere, is naturally an abelian group  under pointwise addition of functions. We denote by $(K^I,+)$ this group, by a slight abuse of notation. We would like to study the 
 translation by $f$ in $K^I$, but as it lacks a nice topology, we consider a compactification of the group $(K^I,+)$, which will be a monoid, as follows.\\

 Let $\mathcal{M}_\mathcal{L}$ be the space of probability measures on $I\times K$ which project to $\mathcal{L}$ on the first factor. 
 To an element $f\in K^I$, we can associate the probability measure $\mathcal{D}_f$ on $I\times K$, supported by its graph, which is the pushforward of $\mathcal{L}$ by the map $x\in I\mapsto (x,f(x))\in I\times K$. This defines an embedding of $K^I$ into $\mathcal{M}_\mathcal{L}$. It is not hard to see that the group law $+$ on $K^I$ correspond to a fiber-wise convolution product $*$ on $\mathcal{M}_\mathcal{L}$, which turns $(\mathcal{M}_\mathcal{L},*)$ into an abelian monoid, with neutral $\mathcal{D}_0$, where $0:I\to K$ is the zero map.\\
 
 The space $\mathcal{M}_\mathcal{L}$ is equipped naturally with a weak-* topology, for which it is a compact metric space. A tricky fact is that the convolution product $(\mu,\nu)\mapsto \mu*\nu$ is {\em not } continuous of the two variables, but is of each variable separately. A more detailed description  of these objects, and explanations of the implied claims, are given in Section \ref{description}.\\
 
 There is a particular element in $\mathcal{M}_\mathcal{L}$, the measure $\lambda=\mathcal{L}\otimes Haar_{K}$. It satisfies the relation: $\forall \mu \in \mathcal{M}_\mathcal{L}$, $\mu*\lambda=\lambda$.\\

 We are interested in studying the dynamics of translation $\mathcal{D}_f*$ on the monoid $\mathcal{M}_\mathcal{L}$. Unsurprisingly, we now need a kind of irrationality condition.
 
\begin{defi} 
  The angle map $f:I\to K$ is said to be {\em asynchronous} if the image measure $f_*\mathcal{L}$ gives zero mass to any translate of any proper closed subgroup of $K$. Equivalently, for any non-trivial character $\chi \in \hat{K}$, $(\chi \circ f)_*\mathcal{L}$ has no atoms. 
\end{defi}
  
  Intuitively, for $d=1$, this means that one looks at an action by rotation on a family of circles indexed by $x\in I$, by angles $f(x)$, which are different from one another if picked randomly following the probability $\mathcal{L}$. For $d\geq 1$, it means that for almost every couple $(x,y)$, $f(x)$ and $f(y)$ do not belong to the same coset modulo any closed, strict subgroup of $K$.\\
   
\begin{thm} \label{mainthm} The following are equivalent.
\begin{enumerate}
\item The angle map $f:I\to K$ is asynchronous. 
\item The closure $\overline{ \{\mathcal{D}_{nf}\} }_{n\in \Z} $ contains $\lambda$.
\item The convolution action of $\mathcal{D}_f$ on $\mathcal{M}_\mathcal{L}$ is uniquely ergodic.
\end{enumerate}
 If these are true, the only invariant probability measure is the Dirac measure $\delta_{\lambda}$.
\end{thm} 

 The fact that the invariant measure is a Dirac measure implies (see Proposition \ref{densityone}) that there exists a subset of the integers $E\subset \Z$, of natural density $1$, such that for any $\mu \in \mathcal{M}_\mathcal{L}$, 
$$\lim_{n \to \pm \infty, \, n\in E} \;  \mathcal{D}_{nf} * \mu = \lambda.$$

 The question whether $\lambda$ is an attracting point of the dynamic, that is if 
$$ \lim_{n \to \pm \infty} \; \mathcal{D}_{nf} *\mu = \lambda,$$
 or if this fails along some subsequence of zero density, is more delicate, and its answer depends on $f$.\\

 If $I=[0,1]$ equipped with Lebesgue measure, $d=1$, and $f$ is a $C^2$ map with non-vanishing derivative, then $\lambda$ is an attracting point (Proposition \ref{toymodel}), and there is no exceptional subsequence. The $C^2$ regularity condition is not optimal, as Thomine obtained similar results for $C^1$ maps \cite{Thomine}.\\

 However, for an angle map $f$ which is only measurable, the convolution action of $\mathcal{D}_f$ might behave more like an intermittent map with the neutral fixed point $\lambda$. An example of this phenomenon is the following. Again, let $I=[0,1]$ endowed with the Lebesgue measure $\mathcal{L}$. Let $\nu$ be the (probability) Hausdorff measure of dimension $\log2/\log3$ on the usual Cantor set $C$, viewed as a subset of $K=\mathbf{T}^1$ by identifying $0$ and $1$. Define
  $f:[0,1]\to \mathbf{T}^1,$
  by 
  $$f(x)=\inf \{t\geq 0  \, :  \, \nu(0,t)\geq x\} \; \mathrm{mod} \; 1.$$
  Then $f_*\mathcal{L}=\nu$, and $f(I) \subset C$. Alternatively, $f$ can be defined as the reciprocal, outside of dyadic rationals, of the usual devil's staircase, modulo $1$.
  Since $\nu$ does not have any atom, $f$ is an asynchronous angle map. We claim that the sequence $(\mathcal{D}_{3^kf})_{k\geq 1}$ does not intersect a fixed neighborhood of $\lambda$. Indeed, since $C$ is invariant by the multiplication $\times 3$ on the circle, the graph of $3^kf$ is contained in $I\times C$, so the measure $\mathcal{D}_{3^kf}$ is supported on $I\times C$, a proper compact subset of $I\times K$. This forbids $\mathcal{D}_{3^kf}$ to be close to $\lambda$. Still, by Theorem \ref{mainthm}, subsequences like $(\mathcal{D}_{3^kf})_{k\geq 0}$ are scarce, as the points $(\mathcal{D}_{nf})_{n\in \Z}$ tend to $\lambda$ for a subset of $\Z$ of density one.\\
 
\subsection{Main example, and a related ergodicity result} 
 
 As hinted in the motivational paragraph, asynchronous rotations occur naturally in the context of homogeneous dynamics on torus bundle.\\
 
  More precisely, let $\mathbf{G}_0$ be a connected, semisimple algebraic linear group defined over $\Q$,  $G_0=\mathbf{G}_0(\R)$ the group of its $\R$-points and $\Gamma_0=\mathbf{G}_0(\Z)$ its integer points. By the Borel - Harish-Chandra Theorem, $\Gamma_0$ is a lattice in $G_0$. We will consider invariant measures on $G_0/\Gamma_0$ under some elements $a\in G_0$, under the following assumptions.

 \begin{defi}
 An element $a\in G_0-\{e\}$ is said {\em triangularizable with positive eigenvalues} if for every finite dimensional representation $\alpha$ of $\mathbf{G}_0$ defined over $\Q$, $\alpha(a)$ has only real, positive eigenvalues. 
 \end{defi} 
 
 It is the case, for example, when $G_0$ is the real split form of $\mathbf{G}_0$, meaning the real rank equals the complex rank, and if $a$ is the exponential of a non-zero element of a Cartan subalgebra. It also happens when $a$ is unipotent, but this case is less interesting for our purpose, since by Ratner's Theory, $a$-invariant ergodic measures are algebraic. This hypothesis implicitly rules out the case where $G_0$ is the real compact form of $\mathbf{G}_0$, as it cannot contain such element $a$. 
 
 \begin{defi}  
  Let $\mu$ be a probability measure on $G_0/\Gamma_0$, invariant and ergodic under the action of $a$.
  Such a measure is said to be {\em non-concentrated} if for every $H\subset G_0$ closed algebraic, strict subgroup containing $a$, and every $x\in G_0/\Gamma_0$ such that $Hx\Gamma_0$ is closed, then $\mu(Hx\Gamma_0)<1$.
 \end{defi} 
 
  We now consider a fiber bundle over the probability space $(G_0/\Gamma_0,\mu)$, whose fibers are tori.\\
 
 Let $\rho:\mathbf{G}_0\to \mathbf{GL}(V)$ a representation defined over $\Q$ on a finite-dimensional space $V=\R^d$ endowed with the $\Z$-structure $\Z^d$. We will always assume that $d\geq 2$, and that $\rho$ is irreducible over $\Q$. The semidirect product $G=G_0\ltimes_\rho V$ is endowed with the group law
 $$\forall(g,v)\in G,\forall (h,w)\in G, \; (g,v)(h,w)=(gh, v+\rho(g)w).$$
 
 Up to replacing $\Gamma_0$ with a subgroup of finite index in a way such that $\rho(\Gamma)\subset \mathbf{GL}(d,\Z)$, the set $\Gamma=\Gamma_0\ltimes_\rho \Z^d$ is a subgroup of $G$, and the map 
 $$\pi:G/\Gamma \to G_0/\Gamma_0,$$
  is a torus bundle. Indeed, the lattice of $\{e\}\times \R^d$ (for the action of multiplication on the left) stabilizing a point $(x,v)\Gamma$ is precisely $\{e\}\times\rho(x)\Z^d$, thus the fiber of $\pi$ over $x\Gamma_0$ is the torus $\R^d/\rho(x)\Z^d$. It will be convenient to have measurable coordinates where this fiber bundle is a direct product. \\

  Let $I\subset G_0$ be a measurable fundamental domain for the action of $\Gamma_0$, and put $\mathcal{L}$ the restriction to $I$ of the $\Gamma_0$-invariant lift of $\mu$. As previously, we denote by $K$ the $d$-dimensional torus $\mathbf{T}^d=\R^d/\Z^d$. The map
 $$I\times K\to G/\Gamma,$$
 $$(x,\bar{v})_{I\times K}\mapsto (x,\rho(x)v)\Gamma,$$
is a measurable bijection, that we will use as an identification between $I\times K$ and $G/\Gamma$, the subscript $I\times K$ indicating the coordinates we are using. Likewise, we will identify $G_0/\Gamma_0$ with $I$ and $\mu$ with $\mathcal{L}$.\\

 For $\beta \in V$, the action of $(e,\beta)$ by multiplication on $G/\Gamma$ on the left, can be read in the $I\times K$ coordinates as the map
 $$(x,\bar{v})_{I\times K} \mapsto (x,\bar{v}+\rho(x)^{-1}\beta )_{I\times K},$$
 i.e. it is a rotation by an angle map $f_\beta: I\to K$, with
  $$f_\beta(x)=\rho(x)^{-1}\beta \;  \mathrm{mod} \; \Z^d.$$
 We prove:
 
\begin{thm}\label{algasyn} 
 Assume that $a\in G_0$ is triangularizable with positive eigenvalues, that $\mu$ is an $a$-invariant, non-concentrated, ergodic probability on $G_0/\Gamma_0$, that $\rho$ is irreducible over $\Q$, and $dim(V)>1$. Assume also that $\beta \in V-\{0\}$ is an eigenvector for $\rho(a)$. Then the angle map
 \begin{eqnarray*}
  f_\beta :  I& \to & K \\
   x &\mapsto & \rho(x)^{-1}\beta \; \mathrm{mod} \; \Z^d,
 \end{eqnarray*}  
is asynchronous.
\end{thm}

 Using the identification of $G/\Gamma$ with $I\times K$, we still denote by $\lambda$ the measure on $G/\Gamma$ such that $\pi_* \lambda=\mu$, whose disintegration along each fiber of $\mu$ are the  Haar measures on each tori. 

 Direct application of Theorem \ref{mainthm} gives:

\begin{cor} \label{coroll}
We assume the same hypotheses as Theorem  \ref{algasyn}. Let $\mathcal{M}_\mu$ be the set of probabilities on $G/\Gamma$ projecting onto $\mu$. The action on $\mathcal{M}_\mu$ induced by the left multiplication by $(e,\beta)$ on $G/\Gamma$ is uniquely ergodic, with invariant measure $\delta_\lambda$.
\end{cor}

 Now choose any $\alpha \in V$. The action of left multiplication by $(a,\alpha)$ on $G/\Gamma$ admits the action of $a$ on $G_0/\Gamma_0$ as a factor. A natural question is if the measure $\lambda$, which is invariant, is ergodic with respect to this action. \\
 
 If $\beta$ is any eigenvector of $\rho(a)$, multiplication by $(e,\beta)$ is in some sense moving in some part of the stable, unstable or neutral direction  (depending on the eigenvalue) of the action of $(a,\alpha)$. Theorem \ref{algasyn} and Hopf's argument allows us to prove the following ergodicity result:

\begin{thm} \label{ergodd}
Assume that $a\in G_0$ is triangularizable with positive eigenvalues, that $\mu$ is an $a$-invariant, non-concentrated, ergodic probability on $G_0/\Gamma_0$, and that $\rho$ is irreducible over $\Q$, of dimension $>1$. Choose $\alpha \in V$, then the action by left multiplication by $(a,\alpha)$ on $G/\Gamma$ is ergodic with respect to the invariant measure  $\lambda$. \\
 If we assume moreover that $\rho(a)$ is not unipotent, then the action of $(a,\alpha)$ on $(G/\Gamma,\lambda)$ is weakly mixing if and only if the action of $a$ on $(G_0/\Gamma_0,\mu)$ is weakly mixing, and the same property holds for strong mixing.
\end{thm}

\subsection{Plan of the paper} 

 In Section \ref{description}, we collect some facts about the topology of $\mathcal{M}_\mathcal{L}$. \\
 
 In Section \ref{Asynchronous}, after dealing with the toy example where $f$ is a $C^2$ map, we prove Theorem \ref{mainthm}. \\
 
 In Section \ref{smallest}, we prove that in the algebraic setting, the smallest algebraic subgroup of $G_0$ containing the elements $\gamma \in \Gamma_0$ induced by Poincar\'e recurrence of tha $a$-action, is $G_0$ itself.
This result (Theorem \ref{Zdense}), which is the main ingredient of the proof of Theorem \ref{algasyn}, relies crucially on the non-concentration of $\mu$.\\

 In Section \ref{application}, we prove the asynchronicity of the rotation obtained by the construction in homogeneous dynamics (Theorem \ref{algasyn}).\\
 
 In Section \ref{ergodproof}, we prove the ergodicity of the extension of the action of $a$ on $G_0/\Gamma_0$ (Theorem \ref{ergodd}). This mainly relies on the following easy observation: if an angle map $f:I\to K$ is asynchronous, then for almost every $x$, the action of translation by $f(x)$ on the torus $K$ is (uniquely) ergodic.
 
\subsection{Acknowledgements}
 
 I would like to thank to Jean-Pierre Conze, Serge Cantat, S\'ebastien Gou\"ezel, Barbara Schapira and Damien Thomine for their feedback and comments on the subject.
 
\section{The space $\mathcal{M}_\mathcal{L}$}
\label{description}

\subsection{Topology of $\mathcal{M}_\mathcal{L}$} 

 We recall that $(I,\mathcal{B}(I))$ be a standard measurable space means that is $I$ can be endowed with a complete, separable distance $d_I$, such that $\mathcal{B}(I)$ is the $\sigma$-algebra of its Borel sets. The facts about standard probability spaces we will use are summarized in 
 \cite[Chapter 1.1]{MR1450400}. Choosing such a distance on $I$ defines a topology on the space of probability measures on $I\times K$, and hence on $\mathcal{M}_\mathcal{L}$.\\
 
 Although the weak-* topology of the space of measures on $I\times K$ depends in a strong way on the choice of topology on $I$, it turns out that:
 
\begin{lem} \label{topologie}
The topology induced on $\mathcal{M}_\mathcal{L}$ does not depends on the choice of topology on $I$. 
\end{lem} 
\begin{proof} 
 Let $I_1,I_2$ two complete, separable metric space endowed with probabilities $\mathcal{L}_i$, with a map $\varphi:I_1\to I_2$ an isomorphism such that $\varphi_*\mathcal{L}_1=\mathcal{L}_2$.
 The topologies induced on measures on $I_i\times K$ are generated by the open sets:
 $$\mathcal{U}^{i}(F,\delta,\mu)=\left\{\nu \,:\, \left| \int_{I_i\times K} Fd\mu-\int_{I_i\times K} Fd\nu\right| <\delta \right\},$$
 where $F:I_i\times K\to \C$ is continuous with compact support for the relevant topology on $I_i\times K$.
 
 Denote by $\tilde{\varphi}$ the map $I_1\times K\to I_2\times K$, $\tilde{\varphi}(x,y)=(\varphi(x),y)$.  To show that $\tilde{\varphi}_*:\mathcal{M}_{\mathcal{L}_1} \to \mathcal{M}_{\mathcal{L}_2}$ is a homeomorphism, it is sufficient by symmetry to show its continuity.\\ 
 
 We fix a neighborhood $\mathcal{U}^{2}(F,\epsilon,\tilde{\varphi}_*\mu)$, and wish to show that its preimage contains some neighborhood of the initial point $\mathcal{U}^{1}(G,\delta,\tilde{\varphi}_*\mu)$, for some $G,\delta$.
 
  The map $\varphi$ from $I_1$ to $I_2$ is measurable. By Lusin's Theorem, for every $\delta>0$, there is a compact set $J\subset I_1$, such that $\mathcal{L}_1(J)>1-\delta$, on which $\varphi$ is continuous. Let $F:I_2\times K\to \C$ be continuous, by Tietze-Urysohn's Theorem, there exists a continuous function $G_\delta:I_1\times K \to \C$ which extends the continuous map $F\circ \tilde{\varphi}:J\times K\to \C$. Moreover, since $F$ is bounded, $G_\delta$ can be chosen such that
 $\|G_\delta\|_\infty=\|F\|_\infty$. If $\nu \in \mathcal{M}_{\mathcal{L}_1}$,
$$ \left| \int_{I_2\times K} F d\tilde{\varphi}_*\mu-\int_{I_2\times K} F d\tilde{\varphi}_*\nu\right| \leq 2\delta\| F\|_\infty+  \left| \int_{J\times K} G_\delta d\mu-\int_{J\times K} G_\delta d\nu\right|,$$
because $\mu((I_2-\varphi J)\times K)=\mathcal{L}_2(I_2-\varphi J)=\mathcal{L}_1(I_1-J)<\delta$ and the same holds for $\nu$.
If we choose $\delta>0$ such that $2\delta\| F\|_\infty<\epsilon/2$, we have:
$$ \left| \int_{I_2\times K} F d\tilde{\varphi}_*\mu-\int_{I_2\times K} F d\tilde{\varphi}_*\nu\right| \leq \epsilon/2 +  \left| \int_{I_1\times K} G_\delta d\mu-\int_{I_1\times K} G_\delta d\nu\right|,$$
and therefore, provided that $\delta<\epsilon/2$,
$$\mathcal{M}_{\mathcal{L}_1}\cap \mathcal{U}^{1}(G_\delta,\delta,\mu) \subset \tilde{\varphi}_*^{-1} (\mathcal{U}^{2}(F,\epsilon,\tilde{\varphi}_*\mu) \cap
\mathcal{M}_{\mathcal{L}_2}).$$
As any neighborhood of $\varphi_*\mu$ contains finite intersections of sets of the form  $\mathcal{U}^{2}(F,\epsilon,\varphi_*\mu)$, this implies that $\tilde{\varphi}_*:\mathcal{M}_{\mathcal{L}_1}\to \mathcal{M}_{\mathcal{L}_2}$ is continuous, as required.
\end{proof}

 A corollary of this discussion is that we can assume for example that $I=[0,1]$ and $\mathcal{L}$ is the Lebesgue measure on this interval, endowed with its usual topology. Since in this case, the set of probability measures on $I\times K$ is a compact, separable metric space, it follows that $\mathcal{M}_\mathcal{L}$ is also compact, separable and metric.

\subsection{Graphs and measures}

 For a measurable map $g:I\to K$, we define the {\em graph measure $\mathcal{D}_g,$ of $g$} as the direct image of $\mathcal{L}$ by the map $x\in I\mapsto (x,g(x)) \in I\times K$. Two measurables maps $I\to K$ define the same graph measure if and only if they are equal $\mathcal{L}$-almost everywhere.\\

 Let $\mathcal{G}$ be the set of graph measures, this is a subset of $\mathcal{M}_\mathcal{L}$.
  
\subsection{Disintegration along $\mathcal{L}$}  
  
 Any $\mu \in  \mathcal{M}_\mathcal{L}$ can be disintegrated as a family of measures $(\mu^x)_{x\in I}$, such that for any
   continuous test-function with compact support $F:I\times K\to \C$, 

$$\mu(F)=\int_I \left( \int_K F(x,y)d\mu^x(y)\right)d\mathcal{L}(x).$$
 
 Moreover, the map $x\mapsto \mu^x$ is measurable, and uniquely defined modulo zero sets. See e.g. \cite[Th. 5.8]{MR603625}.\\

\subsection{Convolution product}

 For $\mu_1,\mu_2$ two measures in $\mathcal{M}_\mathcal{L}$, we define the fiberwise convolution product of $\mu_1,\mu_2$ by

 $$\mu_1*\mu_2(F)=\int_{I} \left( \int_{K^2} F(x,y+z) d\mu_1^x(y) d\mu_2^x(z) \right) d\mathcal{L}(x),$$
where $F:I\times K\to \C$ a continuous test-function with compact support. Equivalently, $(\mu*\nu)^x$ is the usual convolution product of $\mu^x$ and $\nu^x$.\\

 The following Lemma, whose proof is left to the reader, summarizes elementary properties of this fiberwise convolution product. 
\begin{lem} \label{lemmerecap} The following holds.
\begin{enumerate} 
\item $$\forall (f,g)\in (K^I)^2, \, \mathcal{D}_{f+g}=\mathcal{D}_{f}*\mathcal{D}_{g}.$$
\item $\mathcal{D}_0$ is the neutral element of the commutative monoid $(\mathcal{M}_\mathcal{L},*)$, where $0:I\to K$ is the map almost everywhere zero.
\item The set of invertible elements for $*$ is $\mathcal{G}$, the set of graph measures.
\end{enumerate}
\end{lem}
 
  Remark that, if $f:[0,1]\to \mathbf{T}^1$, $f(x)=x \; mod  \;1$, then one can check by hand (or see e.g. Proposition \ref{toymodel}) that $\mathcal{D}_{nf}$ tends to $\lambda$ as $n\to \pm \infty$, but  
 $$\mathcal{D}_{nf}*\mathcal{D}_{-nf}=\mathcal{D}_{0}\neq \lambda*\lambda,$$
 so the fiberwise convolution product is not continuous. However, we have:
 
 \begin{lem} For any $\nu\in \mathcal{M}_\mathcal{L}$, the convolution map
 $$*\nu: \mathcal{M}_\mathcal{L} \to \mathcal{M}_\mathcal{L},$$
 $$\mu \mapsto \mu * \nu,$$
 is continuous.
 \end{lem} 
 \begin{proof} It is sufficient to check that the preimage by $*\nu$ of any set of the form $\mathcal{U}(F,\epsilon,\mu*\nu)$, for any $F$, $\epsilon$, $\mu$, contains a set of the form $\mathcal{U}(G,\delta,\mu)$ for some $G$ and some $\delta>0$. Let $\mathcal{U}(F,\epsilon,\mu*\nu)$ be such a neighborhood of $\mu*\nu$, and let $\delta$ such that $\delta(4\|F\|_\infty+1)<1$. As the map $x\mapsto \nu^x$ is measurable, again by Lusin's Theorem, it is continuous on a set $E$ of measure $1-\delta$.
 Define
 $$H(x,y)=\int_K F(x,y+z)d\nu^x(z).$$
It follows from the continuity of $F$ and the continuity of $x\mapsto \nu^x$ that this is a continuous map on $E\times K$, and moreover bounded by $\|F\|_\infty$. 
 Thus it can be extended to a bounded continuous map, say $G$, on $I\times K$, still bounded by $\|F\|_\infty$. Notice that for any $\eta \in \mathcal{M}_\mathcal{L}$,
 $$\int_{E\times K} Fd(\eta*\nu)=\int_E \int_K \left(\int_K F(x,y+z)d\nu^x(z) \right)d\eta^x(y)d\mathcal{L}(x)=\int_{E\times K} G d\eta.$$
   Let $\eta \in \mathcal{U}(G,\delta,\mu)$, then 
\begin{eqnarray*}
 \left| \int_{I\times K} Fd(\eta*\nu)-\int_{I\times K} Fd(\mu*\nu) \right|  
  & \leq  & 2\delta \|F\|_\infty \\
  &  & + \left| \int_{E\times K} Fd(\eta*\nu)-\int_{E\times K} Fd(\mu*\nu) \right| \\
 & \leq & 2\delta \|F\|_\infty +  \left| \int_{E\times K} G d\eta- \int_{E\times K} Gd\mu \right|\\
 & \leq & 4\delta \|F\|_\infty + \delta.\\
\end{eqnarray*} 
 By the choice of $\delta$,this implies that $*\nu \left( \mathcal{U}(G,\delta,\mu) \right) \subset \mathcal{U}(F,\epsilon,\mu*\nu)$, as announced.
 \end{proof}

\section{Asynchronous maps}  
\label{Asynchronous}   
   
\subsection{A simple example}   
   
 As stated in the introduction, if $f$ has enough regularity properties, it turns out that $\lambda$ is the limit point of the dynamic of $*\mathcal{D}_f$ on $\mathcal{M}_\mathcal{L}$. This result will not be used in the sequel.
 \begin{prop} \label{toymodel}
 Assume $f:[0,1]\to \R/\Z$ is a $C^2$-map, such that $f'$ does not vanish. Then for all $\mu \in \mathcal{M}_\mathcal{L}$, $\mathcal{D}_{kf}*\mu \to \lambda$ when $k\to \pm \infty$.
 \end{prop}
 \begin{proof}
 By continuity of $*\mu$, it is sufficient to check that $\mathcal{D}_{kf}$ tends to $\lambda$ as $k$ tends to infinity.
 To do so, compute the Fourier-Stieltjes coefficients   
 $$\hat{\mathcal{D}}_{kf}(n,m)=\int_{[0,1]} e^{2i\pi(nx+mkf(x))}dx.$$
 If $m=0$, this coefficient is $1$ or $0$, depending on whether $n=0$ or not. If $m\neq 0$, we can write
 $$\hat{\mathcal{D}}_{kf}(n,m)=\int_{[0,1]} \frac{e^{2i\pi nx}}{2i\pi mk \,f'(x)} \frac{\partial}{\partial x}\left( e^{2i\pi mk  f(x)}\right) dx,$$
 and integration by parts immediately shows that $\hat{\mathcal{D}}_{kf}(n,m)=O(\frac1{k})$ when $k\to \pm \infty$ with $n,m$ fixed.
 \end{proof}
   
\subsection{Proof of $1) \Rightarrow 2)$}   
   
   Assume that $f$ is asynchronous. We may, and will, assume that $I=\mathbf{T}^1$ endowed with its Haar probability measure $\mathcal{L}$.
   The space $I\times K$ is then a $(d+1)$-dimensional torus, and to check that some $\mathcal{D}_{nf}$ is close to $\lambda$, it is sufficient to show that
   for a finite set of the non-trivial Fourier-Stieljes coefficients of $\mathcal{D}_{nf}$ are close to zero.
   
 \begin{lem} For any nontrivial character $\chi_0$ of $I\times K= \mathbf{T}^{d+1}$, we have
$$\lim_{N\to +\infty} \frac1N \sum_{n=0}^{N-1}  \left| \widehat{\mathcal{D}_{nf}}(\chi_0) \right|^2=0.$$
 \end{lem}   
   \begin{proof}
   For an integer $k\in{\bf Z}$, let $e_k$ be the character of $\mathbf{T}^1$, $e_k(x)=e^{2i\pi kx}$. Any character $\chi_0$ of $I\times K=\mathbf{T}^{n+1}$ can be written uniquely as
   a product $\chi_0(x,y)=e_k(x)\chi(y)$, for some $k \in {\bf Z}$ and every $x\in I$, $y\in K$. We have
   
   $$\frac1N \sum_{n=0}^{N-1}  \left| \widehat{\mathcal{D}_{nf}}(\chi_0) \right|^2= \int_{I^2} e^{2i\pi k(x-x')} 
   \left( \frac1N \sum_{n=0}^{N-1} \chi(f(x)-f(x'))^n
   \right)
   d\mathcal{L}^2(x,x').$$
   
   If $\chi=1$, then $k\neq 0$ since $\chi_0\neq 1$. In this case, we have
$$\frac1N \sum_{n=0}^{N-1}  \left| \widehat{\mathcal{D}_{nf}}(\chi_0) \right|^2= 0,$$
 so the statement is trivial.\\
   
   If $\chi\neq 1$, then by assumption $\chi(f(x)-f(x'))\neq 1$ for almost every $(x,x')$, so  
   $$\lim_{N\to +\infty} \frac1N \sum_{n=0}^{N-1} \chi(f(x)-f(x'))^n=0.$$
   
   Therefore, Lebesgue's dominated convergence Theorem applies and we obtain the desired result.
 \end{proof}
 
 Let $F\subset \hat{\mathbf{T}}^{d+1}-\{1\}$ be a finite subset of nontrivial characters, and $\epsilon>0$. By the previous Lemma, we have 
   $$\lim_{N\to +\infty} \frac1N \sum_{n=0}^{N-1}  \left( \sum_{\chi_0\in F} \left| \widehat{\mathcal{D}_{nf}}(\chi_0) \right|^2 \right)=0,$$
   so there exists $n\geq 0$ such that for all $\chi_0 \in F$, $\left| \widehat{\mathcal{D}_{nf}}(\chi_0) \right|\leq \epsilon$, meaning that $\mathcal{D}_{nf}$ is close to $\lambda$.

 \subsection{Proof of $2) \Rightarrow 3)$}
 Let us check that $\lambda \in \overline{\{\mathcal{D}_{nf} \}}_{n\in \Z}$ implies that $*\mathcal{D}_f$ is uniquely ergodic. Let $(n_i)_{i\geq 1}$ be a sequence such that $\mathcal{D}_{n_if}$ converges weakly to $\lambda$ when $i\to+\infty$. 
 Let $\mu \in \mathcal{M}_\mathcal{L}$, then by one-sided continuity of convolution,
 $$\mu*\mathcal{D}_{n_if}\rightarrow_{i\to+\infty} \mu*\lambda=\lambda.$$
 Let $m$ be any invariant measure on $\mathcal{M}_\mathcal{L}$, $F:\mathcal{M}_\mathcal{L}\to \R$ be a continuous function. Then
 $$\int_{\mathcal{M}_\mathcal{L}} F(\mu)dm(\mu)=\int_{\mathcal{M}_\mathcal{L}} F(\mu*\mathcal{D}_{n_if})dm(\mu),$$
 by invariance of $m$. The Lebesgue dominated convergence Theorem implies
  $$\int_{\mathcal{M}_\mathcal{L}} F(\mu)dm(\mu) \to \int_{\mathcal{M}_\mathcal{L}} F(\lambda)dm(\mu)=F(\lambda),$$
  which means that $m$ is the Dirac measure at $\lambda$, as required.

\subsection{Proof of $3) \Rightarrow 1)$}
\label{converse}

 Assume that the convolution action of $\mathcal{D}_f$ is uniquely ergodic. As $\lambda$ is a fixed point, the invariant measure is necessarily $\delta_\lambda$, and thus as the invariant measure is a Dirac mass, there exist a subsequence $n_i\to +\infty$ such that $\mathcal{D}_{n_i f}$ tends to $\lambda$ as $i\to +\infty$. \\
 
 Let $\chi$ be any non-trivial character of $K$. Assume that $(\chi\circ f)_*\mathcal{L}$ has an atom. In this case, there would be a set $E\subset I$ of positive measure on which $\chi\circ f$ is a constant, say $c$. Thus  
$\chi\circ (n_i f)$ is a constant on $E$, namely $c^{n_i}$, and $(\chi \circ (n_i f))_*\mathcal{L}$ will have an atom of mass $\mathcal{L}(E)$. Note that
$$(\chi \circ (n_i f))_*\mathcal{L}=(\chi \circ \pi_K)_* \mathcal{D}_{n_i f}.$$
But $(\chi \circ \pi_K)_* \mathcal{D}_{n_i f}$ tends to $(\chi \circ \pi_K)_*\lambda$, namely the Lebesgue measure on $\mathbf{T}^1$, which  cannot be a limit of measures having a atom of fixed mass. This is a  contradiction.

\subsection{Sets of natural density one}
 
 \begin{prop} \label{densityone}
  Assume $f:I\to K$ is asynchronous. Then there exists a set $E\subset \Z$ of full natural density  such that for all $\mu \in \mathcal{M}_\mathcal{L}$,
 $$\lim_{n \to \pm \infty, \, n\in E} \;  \mathcal{D}_{nf}* \mu = \lambda.$$
 \end{prop}
 \begin{proof}
  We consider the measure on $\mathcal{M}_\mathcal{L}$,
  $$\nu_N=\frac1{2N+1}\sum_{|k|\leq N} \delta_{\mathcal{D}_{kf}}.$$
  As any weak limit of $\nu_N$ is $\mathcal{D}_f*$-invariant and $\mathcal{M}_\mathcal{L}$ is compact, by unique ergodicity of $(\mathcal{M}_\mathcal{L},\mathcal{D}_f*)$, $\nu_N$ converges to $\delta_{\lambda}$ when $N$ goes to $+\infty$. This implies that for any neighborhood $\mathcal{U}$ of $\lambda$, the proportion of $\{\mathcal{D}_{kf} \}_{|k|\leq N}$ outside $\mathcal{U}$ goes to zero as $N\to +\infty$.\\
  Let $(\mathcal{U}_m)_{m\geq 1}$ be a decreasing basis of neigborhood of $\lambda$, and
  $$E_m=\{ k\in \Z \; : \; \mathcal{D}_{kf}\in \mathcal{U}_m \}.$$
  Let $N_m$ be an integer such that for all $N\geq N_m$,
  $$\mathbf{P}_N(E_m)\geq 1-\frac{1}{m},$$
  where $\mathbf{P}_N$ is the uniform probability on $[-N,N]$. We can modify the sequence $(N_m)_{m\geq 0}$ to be strictly increasing, and choose $N_1=-1$.
  Let $E$ be the subset 
  $$E=\bigcup_{m\geq 1} E_m\cap \{ k\in \Z \, : \, N_m<|k|\leq N_{m+1}\}.$$
  Notice that since the sets $E_m$ are decreasing with $m$, if $n\leq  N_{m+1}$,
  $$E \cap  [-n,n]\subset E_m.$$
  Thus, for $n$ such that $N_m<n\leq N_{m+1}$, we have
  $$\mathbf{P}_n(E)\geq 1-\frac{1}{m}.$$
  This proves that $E$ is a set of natural density one. By construction, we have 
  $$\lim_{n \to \pm \infty, \, n\in E} \; \mathcal{D}_{nf} = \lambda.$$
  By continuity of the convolution with $\mu$, the latter limit holds for the sequence $\mathcal{D}_{nf}*\mu$ with the same set $E$.
 \end{proof}

\section{On the smallest algebraic group containing return elements}
\label{smallest}

 The following Theorem, which will be a crucial ingredient of the proof of Theorem \ref{algasyn}, might be of independent interest.

 \begin{thm} \label{Zdense}
  Let $G_0$ be the group of real points of an algebraic group $\mathbf{G}_0$ defined over $\Q$, without nontrivial $\Q$-characters, $\Gamma_0=\mathbf{G}_0(\Z)$ be its integer points, $a\in G_0$ be triangularizable with positive eigenvalues, and $\mu$ an $a$-invariant measure
 on $G_0/\Gamma_0$.
 We assume that the measure $\mu$ is ergodic and non-concentrated. Let $I\subset G_0$ be a fundamental domain for $\Gamma_0$, and denote by $\mathcal{L}$ the lift to $I$ of $\mu$. Let $E\subset I$ be a subset of positive $\mathcal{L}$-measure. Define
 $$P_E=\left\{\gamma \in \Gamma_0 \, : \, \mathcal{L}\left( \cup_{k\in \Z} a^k E\gamma\cap E \right)>0 \right\},$$
 the set of elements of $\Gamma_0$ associated to return times in $E$. Then the smallest algebraic subgroup of $G_0$ containing $P_E$ is $G_0$.
 \end{thm}
 
  To prove this, let $H$ be the smallest algebraic subgroup of $G_0$ containing $P_E$. Our aim is to show that $H=G_0$. This will be done in the following sequence of Lemmata.
  
\subsection{Closure of $H\Gamma_0$}  
  
  \begin{lem} \label{Hisclosed}
   The set $H\Gamma_0$ is closed
  \end{lem}
  \begin{proof}
 Notice that $H$ is defined over $\Q$, since $P_E$ consists of integer points. We claim that the the non-trivial $\Q$-characters of $H$ are of order $2$. Indeed, if $c$ is such a character defined over $\Q$, the image by $c$ of the subgroup generated by $P_E \cap H^0\subset \Gamma_0=\mathbf{G}(\Z)$ consists of rational with bounded denominators, and is a multiplicative subgroup, so $c(P_E)\subset \{-1,+1\}$. Therefore, $P_E$ is contained in $Ker(c^2)$, an algebraic group defined over $\Q$. By definition of $H$, $H\subset Ker(c^2)$, so $H=Ker(c^2)$ as required. In particular, $H/(H\cap \Gamma_0)$ is of finite volume, by the Theorem of Borel and Harish-Chandra \cite[Corollaire 13.2]{MR0244260}.
By \cite[Proposition 8.1]{MR0244260}, this also implies that $H^0\Gamma_0$ is a closed subset of $G_0/\Gamma_0$, where $H^0$ is the connected component of the identity of $H$, in the Zariski topology (a subgroup of finite index). This implies that $H\Gamma_0$ is closed.
  \end{proof}

\subsection{Reduction step} 
 
 \begin{lem} \label{assumption4}
 To prove Theorem \ref{Zdense}, we can (and will) assume that for all $k\in \Z$ and $\gamma \in \Gamma_0$ such that $a^k E\gamma \cap E\neq \emptyset$, then $\gamma \in P_E$.
 \end{lem} 
 \begin{proof} 
 Consider the subset     
 $$F= E - \bigcup_{(k,\gamma)\in \Z\times \Gamma_0 \; \mathbf{s.t.} \; \mathcal{L}(a^kE\gamma\cap E)=0} a^k E\gamma.$$
 Clearly, $F$ is a subset of $E$ of the same measure, and $P_F=P_E$. So it is sufficient to prove the statement of Theorem \ref{Zdense} for $F$ instead of $E$, and $F$ satisfies the above property.
 \end{proof}
  
\subsection{Invariance of $xH\Gamma_0$}  
  
 \begin{lem} For $\mathcal{L}$-almost every $x\in E$,  $a \in xHx^{-1}$.
 \end{lem}
 \begin{proof}
  By a Theorem of Chevalley \cite[Thm 5.1]{MR1102012}, there exists a finite dimensional representation $\alpha$ of $\mathbf{G}$ such that $H$ is the stabilizer of a line $D$, that is $H=\{g \in G_0 \, : \, \alpha(g)D=D\}$. By Poincar\'e recurrence Theorem, for $\mathcal{L}$-almost every $x \in E$, there exists a sequence $n_k\to +\infty$ and $\gamma_k \in \Gamma_0$ such that $a^{n_k}x\gamma_k \to x$ and $a^{n_k} x \gamma_k \in E$. Fix such an $x\in E$.\\
  
  By Lemma \ref{assumption4}, we known that $\gamma_k \in P_E\subset H$. It follows that $\alpha(\gamma_k)D=D$, so since $a^{n_k}x\gamma_k \to x$, we have
\begin{equation}\label{eqstab}
\lim_{k\to +\infty} \alpha(a)^{n_k}\alpha(x)D=\alpha(x)D.
\end{equation}

  By assumption, $a$ is triangularizable with positive eigenvalues, so $\alpha(a)$ has only positive, real eigenvalues. We claim that \eqref{eqstab} implies that
  $\alpha(x)D$ is contained in one of its eigenspaces.  
    
 Let  
 $$\alpha(a)=\delta+\eta,$$
 be the Jordan-Chevalley decomposition of $\alpha(a)$, that is: $\delta$ and $\eta$ commutes, $\eta$ is nilpotent, $\delta$ diagonalizable (with positive, real eigenvalues).
 If $p$ is the nilpotent index of $\eta$, for $k$ such that $n_k>p$, 
 $$\alpha(a)^{n_k}= \sum_{i=0}^{p-1} \binom{n_k}{i} \delta^{n_k-i}\eta^i.$$
 Let $v\in \alpha(x)D-\{0\}$, and $v=\sum_{\theta} v_\theta$ be its decomposition along the eigenspaces of $\delta$ corresponding to the eigenvalues $\{\theta\}$ of $\delta$. Then
  $$\alpha(a)^{n_k}(v) = \sum_\theta \left(\sum_{i=0}^{p-1} \binom{n_k}{i} \theta^{n_k-i}\eta^i(v_\theta) \right).$$
 As a function of $n_k$, this is a combination of polynomials and powers of eigenvalues. 
 If $\theta_0$ is the highest eigenvalue $\theta$ for which $v_\theta\neq 0$, and $i_0$ is the largest $i$ for which $\eta^i(v_{\theta_0})\neq 0$, then we have the asymptotic as $k\to +\infty$,
  $$\alpha(a)^{n_k}(v) \sim \binom{n_k}{i_0} \theta^{n_k-i_0}\eta^{i_0}(v_{\theta_0}).$$
  
 However, we know that projectively, $\alpha(a)^{n_k}\alpha(x)D \to \alpha(x)D$, so $v$ is colinear to $\eta^{i_0}(v_{\theta_0})$. 
As $\eta$ preserves the eigenspace of $\delta$ associated to $\theta_0$, 
$$\alpha(a)\eta^{i_0}(v_{\theta_0})=\theta_0 \eta^{i_0}(v_{\theta_0}) + \eta^{i_0+1}(v_{\theta_0})=\theta_0 \eta^{i_0}(v_{\theta_0}),$$
because by definition of $i_0$, $\eta^{i_0+1}(v_{\theta_0})=0$. This shows that $\eta^{i_0}(v_{\theta_0})$ is an eigenvector, and so is $v$.

 We have proved that $\alpha(x)D$ is contained in an eigenspace of $\alpha(a)$. So $D$ is stabilized by $\alpha(x^{-1}ax)$, meaning that $a \in xHx^{-1}$, as required.
 \end{proof}
 
\subsection{Conclusion of the proof of Theorem \ref{Zdense}} 
 
\begin{lem} \label{HequalG}
We have $H=G_0$.
\end{lem} 
 \begin{proof}
  By ergodicity of $a$ with respect to $\mu$, for $\mu$-almost every $x\Gamma_0$, $a^{\Z}x\Gamma_0$ is dense in the support of $\mu$.
  By the previous Lemma, we have also for almost every $x\in E$, $a \in xHx^{-1}$, so 
  $$a^{\Z}x\Gamma_0\subset xH\Gamma_0.$$
  Consider a typical $x \in E$ satisfying both of these properties.
  By Lemma \ref{Hisclosed}, $xH\Gamma_0$ is a closed set. By density of the $a$-orbit of $x$ in the support of $\mu$, this implies that $supp(\mu)\subset xH\Gamma_0$, so
  $$\mu(xH\Gamma_0)=\mu\left( (xHx^{-1}) x\Gamma_0\right)=1.$$
  By assumption, $\mu$ is non-concentrated, so $H=G_0$ necessarily.  
 \end{proof}

\section{Proof of Theorem \ref{algasyn}}
\label{application}

 The proof is by contradiction. We assume that $f_\beta$ is not asynchronous. By translate of a $\Q$-subspace of $V=\R^d$, we mean a set $T\subset V$ of the form $T=v+W$, where $v\in V$, and $W\subset V$ a subspace of $V$ defined over $\Q$; in particular, there is no rationality assumption on $v$, and $T$ itself does not have to be defined over $\Q$. The $\Q$-subspace $W$ is called the direction of $T$. \\
 
 If non-empty, the intersection of two translates of $\Q$-subspaces is again a translate of $\Q$-subspace. This property allows us to define, for
 a set $E\subset I$ of positive $\mathcal{L}$-measure, the set $T_E$ which is the smallest translate of $\Q$-subspace containing 
 $\{\rho(x)^{-1}\beta\}_{x\in E}$. We denote by $W_E$ its direction.\\
 
 \begin{lem} \label{notall}
 There exist $E\subset I$ of positive measure such that $W_E\neq V$.
 \end{lem}
 \begin{proof}    
  Since $f_\beta$ is not asynchronous, so for some non-trivial character $\chi\in \hat{K}$, $\chi\circ f_\beta$ is constant on a set $F$ of positive $\mathcal{L}$-measure. There exists $\mathbf{n}\in \Z^d-\{0\}$, such that
  $$\chi(\bar{v})=e^{2i\pi \, \langle \mathbf{n}, v\rangle}.$$
  So the set $\{\langle \mathbf{n},  \rho(x)^{-1}\beta \rangle \}_{x\in F}$ is contained in a countable set of the form $c+\Z$, for some $c\in \R$.
 This implies that at least one of the sets
 $$F_m  =\{ x \in F \; : \; \langle \mathbf{n},  \rho(x)^{-1}\beta \rangle=c+m \},$$
 for $m\in \Z$, has positive $\mathcal{L}$-measure. By construction, for such a $m\in \Z$, 
 $$T_{F_m}\subset \{ v \in V \; : \; \langle \mathbf{n},  v \rangle = c+m\},$$
 the right-hand side set being the translate of a proper $\Q$-subspace, $E=F_m$ satisfies the Lemma.
 \end{proof}
 
  We now fix $E\subset I$ a set of positive measure, such that $W_E$ is of minimal possible dimension (it exists). By Lemma \ref{notall}, $W_E\neq V$.  
  Like in Theorem \ref{Zdense}, we define
   $$P_E=\left\{\gamma \in \Gamma_0 \, : \, \mathcal{L}\left( \cup_{k\in \Z} a^k E\gamma\cap E \right)>0 \right\}.$$

  \begin{lem} \label{stability}
For all $\gamma \in P_E$, $\rho(\gamma)W_E=W_E$.
  \end{lem}
 \begin{proof} 
  Recall that $\beta$ is an eigenvector for $\rho(a)$, denote by $\kappa$ the corresponding eigenvalue.
  By definition of $P_E$, there exists $k\in \Z$ with $\mathcal{L}\left( a^k E\gamma\cap E \right)>0$. Let $F=a^k E\gamma\cap E$, then for all $x\in F$,
  there exist $y \in E$ such that $x=a^k y \gamma$. We have: 
  $$\rho(x)^{-1}\beta=\rho(\gamma^{-1}y^{-1}a^{-k}) \beta=\kappa^{-k} \rho(\gamma)^{-1} \rho(y)^{-1} \beta \in \kappa^{-k} \rho(\gamma)^{-1} T_E.$$
  Note that since $\rho(\gamma)^{-1}$ is a matrix with integer coefficients,  $\kappa^{-k} \rho(\gamma^{-1}) T_E$ is also the translate of a $\Q$-subspace, containing $\{\rho(x)^{-1}\beta\}_{x\in F}$. By definition of $T_F$, this means that 
  $$T_F \subset \kappa^{-k} \rho(\gamma)^{-1} T_E.$$
   Since $F\subset E$, $T_F\subset T_E$, and because $E$ was chosen such that $T_E$ is of minimal possible dimension, we have $T_E=T_F$, so from the rank-nullity Theorem,
  $$\kappa^{-k} \rho(\gamma)^{-1} T_E  = T_E.$$  
  This implies equality for the directions, $\rho(\gamma)^{-1} W_E  = W_E$, and multiplication by $\rho(\gamma)$ concludes the proof.  
 \end{proof}

 \begin{lem} We have $W_E=\{0\}$, that is, the map $x\mapsto \rho(x)^{-1}\beta$ is constant on $E$.
 \end{lem}
 \begin{proof}
  The subgroup 
$Stab_{G_0}(W_E)=\{g\in G_0 \; : \; \rho(g)W_E=W_E \}$ is an algebraic subgroup containing $P_E$, by Lemma \ref{stability}. By Theorem \ref{Zdense},
this group is $G_0$. Since $\rho$ is irreducible over $\Q$ and $W_E$ is defined over $\Q$, $W_E=\{0\}$, or  $W_E=V$. But the latter cannot happen, because of the choice of $E$.
 \end{proof}
 
 From now on, we fix some $x_0 \in E$. By the previous Lemma, $T_E$ is the point $\rho(x_0)^{-1}\beta$.
  
 \begin{lem}  For any $\gamma \in P_E$, $\rho(\gamma) \in Stab(\R \rho(x_0)^{-1}\beta)$.
 \end{lem}
 \begin{proof}
  Let $\gamma \in P_E$. Thus there exist $k \in \Z$ such that $\mathcal{L}(a^kE\gamma\cap E)>0$.
  In the proof of Lemma \ref{stability}, we saw that 
  $$\kappa^{-k} \rho(\gamma)^{-1} T_E  = T_E.$$  
  But since $T_E=\{\rho(x_0)^{-1}\beta\}$, this means that $\rho(\gamma)^{-1}$ stabilizes the line through $\rho(x_0)^{-1}\beta$.
 \end{proof} 

  The end of the proof of Theorem \ref{algasyn} is given by the following contradictory Claim.
 
 \begin{lem} The space $V$ is one-dimensional.
 \end{lem}
 \begin{proof}
  The group 
  $$\{g\in G \; : \; \rho(g) \in Stab(\R \rho(x_0)^{-1}\beta) \},$$
  is an algebraic group containing $P_E$. By Theorem \ref{Zdense}, it follows that $\rho(x_0)^{-1}\beta$ is a common eigenvector for all elements of $\rho(G_0)$ (and so is $\beta$). Were the representation $\rho$ irreducible over $\R$, this would be sufficient to conclude; however we assumed only $\Q$-irreducibility, and have no particular rationality assumption on $\beta$.\\
  
  Since $G_0$ is semisimple and connected, the eigenvalue associated to $\rho(x_0)^{-1}\beta$ is $1$ for every $g\in G_0$. Let $V_1(g)$ denote the eigenspace associated to the eigenvalue $1$ for the operator $\rho(g)$. Let $Z=\cap_{\gamma \in P_E} V_1(\gamma)$. This subspace $Z$ is defined over $\Q$, because $P_E$ consists of integral points. The set of $g\in G_0$ acting trivially on $Z$ is an algebraic subgroup containing $P_E$, so again is $G_0$. Since $\rho(x_0)^{-1}\beta\in Z$, $Z$ is of positive dimension. By $\Q$-irreducibility of $\rho$, $Z=V$ is an irreducible representation where $G_0$ acts trivially, so is one-dimensional.
 \end{proof}

\section{Ergodicity and mixing} \label{ergodproof}

 We now prove Theorem \ref{ergodd}. By assumption, $a$ is triangularizable with positive eigenvalues. We separate the proof in two cases.\\
 
\noindent \underline{Case 1}: $\rho(a)$ is unipotent.\\

 Let $\beta \in V-\{0\}$ be an eigenvector for $\rho(a)$, its eigenvalue is $1$.
 
Notice that the actions of $(e,\beta)$ and $(a,\alpha)$ commute: since $\rho(a)\beta=\beta$, we have
 $$(a,\alpha)(e,\beta)=(a,\alpha+\rho(a)\beta)=(a,\rho(e)\alpha+\beta)=(e,\beta)(a,\alpha).$$

Consider the ergodic decomposition of $\lambda$ with respect to the action of $(a,\alpha)$: there exists a measure $m$ on the set of ergodic, $(a,\alpha)$-invariant measures on $G/\Gamma$, such that 
$$\lambda=\int \nu\, dm(\nu).$$
If we apply the projection map $\pi : G/\Gamma \to G_0/\Gamma_0$ to this equality, we obtain
$$\mathcal{L}=\mu=\int (\pi)_* \nu \, dm(\nu),$$
where $(\pi)_*\nu$ are $a$-invariant. Since $\mu$ is $a$-ergodic, we have that $(\pi)_*\nu=\mu$, $m$-almost surely. Therefore, $m$ is supported on $\mathcal{M}_\mathcal{\mu}$. But since $(e,\beta)$ commutes with $(a,\alpha)$, $(e,\beta)_* m$ is the measure associated to the ergodic decomposition of
$(e,\beta)_*\lambda=\lambda$. This implies that $m$ is $(e,\beta)$-invariant, and by Corollary \ref{coroll}, $m$ is the Dirac measure at $\lambda$.
This concludes the proof of the ergodicity.\\
 
\noindent \underline{Case 2}: $\rho(a)$ is not unipotent. Thus $\rho(a)$ has some of its eigenvalues different from $1$. Since $G_0$ is semisimple, $det(\rho(a))=1$ so there exist at least one eigenvalue $\kappa<1$. Let $\beta\in V$ be an eigenvector of $\rho(a)$ associated to $\kappa$. \\

 Notice that , for $k\geq 0$,
 $$ (a,\alpha)^k (e,\beta) =(a^k, \rho(a)^k\beta+  \sum_{i=0}^{k-1} \rho(a)^i\alpha)=(e,\kappa^k \beta)(a,\alpha)^k,$$
 and $(e,\kappa^k \beta)\to (e,0)$ when $k\to +\infty$. This implies that the distance (with respect to a distance $d_{G/\Gamma}$ on $G/\Gamma$ induced by a right-$G$-invariant riemannian distance on $G$)  between $(a,\alpha)^k (x,v)\Gamma$ and $(a,\alpha)^k (e,\beta) (x,v)\Gamma$ tends to zero as $k$ tends to $+\infty$. In other words, the strong stable distribution for $(a,\alpha)$, defined by
 $$W^{ss}((x,v)\Gamma)=\{(y,w)\Gamma \; : \; \lim_{k\to +\infty} d_{G/\Gamma}((a,\alpha)^k(x,v)\Gamma,(a,\alpha)^k(y,w)\Gamma)= 0 \},$$
 is invariant under the action of $(e,\beta)$.\\
 
 We first prove the claims about ergodicity and weak-mixing. \\
 
  Let $f \in L^2(G/\Gamma,\lambda)-\{0\}$ be an eigenvector for the Koopman operator of $(a,\alpha)$, that is
  $$f((a,\alpha)(x,v)\Gamma)=\omega  f((x,v)\Gamma),$$
  for $\lambda$-almost every $(x,v)\Gamma$, for some $\omega \in \C$ of modulus one. Recall that ergodicity states that any such eigenvector associated to $\omega=1$ is constant almost everywhere, and weak-mixing that any such eigenvector is constant almost everywhere and moreover $\omega=1$.\\
  
  To prove ergodicity or weak-mixing, we may (and will) restrict to the case where $f$ is bounded. By the Hopf argument, and more precisely the version proved by Coud\`ene \cite{MR2261076}, $f$ is $W^{ss}$-invariant, meaning that there exists a full measure set $\Omega \subset G/\Gamma$ such that for all $(x,v)\Gamma \in \Omega,(y,w)\Gamma)\in \Omega$ such that $(y,w)\Gamma \in W^{ss}((x,v)\Gamma)$, $f((x,v)\Gamma)=f((y,w)\Gamma)$.\\
  
  By Theorem \ref{algasyn}, the angle map $x \mapsto \rho(x)^{-1}\beta$ is asynchronous. In particular, for every strict linear rational (closed) subtorus $\mathcal{T}\subset K$, the set of $x$ such that $\rho(x)^{-1}\beta$ does not belong to $\mathcal{T}$ is of full measure. As the set of such subtorus is countable, this implies that for $\mu$-almost every $x$, the translation on $K$ given by $\rho(x)^{-1}\beta$ is ergodic. Via a linear change of variable, this means that for $\mu$-almost every $x\Gamma_0 \in G_0/\Gamma_0$, the translation by $(e,\beta)$ on the fiber above $x\Gamma_0$ is ergodic.\\
  
  As $f$ is $(e,\beta)$-invariant, for $\mu$-almost every $x\Gamma_0$, $f$ is almost everywhere constant on the fiber above $x\Gamma_0$, and merely depends on $x\Gamma_0$. Write $F \in L^2(G_0/\Gamma_0,\mu)$ for its almost everywhere value, that is $f=F\circ \pi$, $\lambda$-almost everywhere, where $\pi:G/\Gamma \to G_0/\Gamma_0$ is the fiber bundle. We have then
 $$F(ax\Gamma_0)=\omega  F(x\Gamma_0).$$
 
 If $\omega=1$, then $F$ is $a$-invariant, and by ergodicity of $a$, $F$ is constant $\mu$-a.e., and so $f$ is constant $\lambda$-a.e. . This proves the ergodicity.\\
 
 Assume now that the action of $a$ is weakly mixing on $(G_0/\Gamma_0,\mu)$. Let $f\in L^2(G/\Gamma,\lambda)$ be like previously an eigenvector for the Koopman operator of $(a,\alpha)$, $F \in L^2(G_0/\Gamma_0,\mu)$ its almost sure value depending on the fiber. By what we saw before, $F$ is an eigenvector for the Koopman operator of $a$, so $\omega=1$ by weak-mixing of $a$. By ergodicity, $F$ is constant, as was to be proved.\\

 We now assume that the action of $a$ on $(G_0/\Gamma_0,\mu)$ is strongly mixing, and wish to prove that $(a,\alpha)$ is also strongly mixing. Recall that strong mixing of $a$ is equivalent to the fact that for all $F\in L^2(G_0/\Gamma_0,\mu)$, $F\circ a^k$ converges weakly to a constant as $k\to +\infty$.\\
 
 Let $f \in L^2(G/\Gamma,\lambda)$. Let $g \in L^2(G/\Gamma,\lambda)$ be any weak limit of $f\circ (a,\alpha)^k$  as $k\to +\infty$ along a subsequence. 
 By another result of Coud\`ene \cite{MR2297089}, generalizing a result of Babillot, $g$ is $W^{ss}$-invariant. By ergodicity of $(e,\beta)$ on almost every fiber, $g=G\circ \pi$ almost surely, where $G\in L^2(G_0/\Gamma_0,\mu)$. Define
 $$F(x\Gamma_0)=\int_{\R^d/\rho(x)\Z^d} f((x,v)\Gamma) dHaar_{\R^d/\rho(x)\Z^d}(v),$$
 the mean value of $f$ on each fiber. Let $H \in L^2(G_0/\Gamma_0,\mu)$ be a test-function. Then 
 
 $$\int_{G/\Gamma} (f\circ (a,\alpha)^k)    (H \circ \pi) d\lambda=\int_{G_0/\Gamma_0} (F\circ a^k)  . H  d\mu.$$
Taking the limits in the left-hand side and right-hand side respectively along the subsequence, using the strong mixing property for $a$, gives:
$$\int_{G/\Gamma} (G\circ \pi) (H\circ \pi) d\lambda = \left(\int_{G_0/\Gamma_0} F d\mu\right) \left(\int_{G_0/\Gamma_0} H d\mu\right),$$
in other words,
$$\int_{G_0/\Gamma_0} G\, H d\mu = \left(\int_{G_0/\Gamma_0} F d\mu\right) \left(\int_{G_0/\Gamma_0} H d\mu\right),$$
which implies that $G$ is $\mu$-almost everywhere the constant $\int_{G_0/\Gamma_0} F d\mu=\int_{G/\Gamma} f d\lambda$. Therefore the only possible weak limit of $f \circ (a,\alpha)^k$ is the above constant. By weak compactness of the ball of radius $\|f\|_ 2$ in $L^2(G/\Gamma,\lambda)$, this proves that this sequence must converge weakly to $\int_{G/\Gamma} f d\lambda$, as required.

\bibliography{biblio}{}
\bibliographystyle{plain}

\end{document}